\documentclass{article}

\usepackage{amsmath,amssymb,amsfonts,amsthm}
\usepackage{authblk,proof,graphicx,tikz}
\usetikzlibrary{tikzmark,calc}

\newcommand{\mul}{\cdot}
\newcommand{\U}{\mathbf{1}}
\newcommand{\Z}{\mathbf{0}}
\newcommand{\INC}{\text{\sc inc}}
\newcommand{\JZDEC}{\text{\sc jzdec}}

\newcommand{\ra}{\mathsf{a}}
\newcommand{\rb}{\mathsf{b}}
\newcommand{\rc}{\mathsf{c}}
\newcommand{\Mf}{\mathcal{M}}
\newcommand{\Nf}{\mathcal{N}}
\newcommand{\NN}{\mathbb{N}}
\newcommand{\imp}{\multimap}
\newcommand{\pmi}{\mathop{\mbox{\rotatebox{180}{\raisebox{-5pt}{$\multimap$}}}}}
\newcommand{\gn}[1]{\ulcorner {#1} \urcorner}
\newcommand{\Ac}{\mathcal{A}}
\newcommand{\Hc}{\mathcal{H}}
\newcommand{\Kc}{\mathcal{K}}
\newcommand{\Cc}{\mathcal{C}}

\newcommand{\CommACT}{\mathbf{CommACT}}
\newcommand{\CommACTomega}{\CommACT_\omega}
\newcommand{\ACT}{\mathbf{ACT}}
\newcommand{\ACTomega}{\ACT_\omega}
\newcommand{\CommACTinfty}{\CommACT_\infty}

\newcommand{\Var}{\mathrm{Var}}
\newcommand{\yields}{\vdash}

\newtheorem{theorem}{Theorem}
\newtheorem{prop}{Proposition}
\newtheorem{lemma}{Lemma}
\theoremstyle{definition}
\newtheorem*{definition}{Definition}
\theoremstyle{remark}
\newtheorem{example}{Example}

\begin{document}
 
 \title{Commutative Action Logic}
 \author{Stepan L. Kuznetsov}
 \affil{Steklov Mathematical Institute of RAS}
 
 \maketitle 
 
 \begin{abstract}
  We prove undecidability and pinpoint the place in the arithmetical hierarchy for commutative action logic, that is, the equational theory of commutative residuated Kleene lattices (action lattices), and infinitary commutative action logic, the equational theory of *-continuous action lattices. Namely, we prove that the former is $\Sigma_1^0$-complete and the latter is $\Pi_1^0$-complete. Thus, the situation is the same as in the more well-studied non-commutative case. The methods used, however, are different: we encode infinite and circular computations of counter (Minsky) machines.
 \end{abstract}

 \section{Action Lattices and Their Theories}
 
 The concept of {\em action lattice,} introduced by Pratt~\cite{Pratt1991} and Kozen~\cite{Kozen1994LIF}, combines several algebraic structures: a partially ordered monoid with
residuals (``multiplicative structure''), a lattice (``additive structure'') sharing the same partial order, and Kleene star. (Pratt introduced the notion of action algebra, which bears only a semi-lattice structure with join, but not meet. Action lattices are due to Kozen.)

\begin{definition}
 An action lattice is a structure 
 $\langle \Ac; \preceq, \mul, \Z, \U, \imp, \pmi, \vee, \wedge, {}^* \rangle$, where:
 \begin{enumerate}
  \item $\preceq$ is a partial order on $\Ac$;
  \item $\Z$ is the smallest element for $\preceq$, that is, $\Z \preceq a$ for any $a \in \Ac$;
  \item $\langle \Ac; \mul, \U \rangle$ is a monoid;
  \item $\imp$ and $\pmi$ are residuals of the product ($\cdot$) w.r.t.\ $\preceq$, that is:
  $$
b \preceq a \imp c \iff a \mul b \preceq c \iff 
a \preceq c \pmi b;
  $$
  \item $\langle \Ac; \preceq, \vee, \wedge \rangle$ is a lattice;
  \item for each $a \in \Ac$, $a^* = \min_{\preceq} \{ b \mid \U \preceq b \mbox{ and } a \mul b \preceq b \}$.
 \end{enumerate}

\end{definition}

An important subclass of action lattices is formed by 
{\em *-continuous} action lattices.

\begin{definition}
 Action lattice $\Ac$ is *-continuous, if for any $a \in \Ac$ we have $a^* = \sup_{\preceq} \{ a^n \mid n \ge 0 \}$, where $a^n = a \mul \ldots \mul a$ ($n$ times) and
 $a^0 = \U$.
\end{definition}

Interesting examples of action lattices are mostly *-continuous; non-*-con\-ti\-nuous action lattices also exist, but are constructed artificially.

The equational theory for the class of action lattices or its subclass (e.g., the class of *-continuous action lattices) is the set of all statements of the form $A \preceq B$, where $A$ and $B$ are formulae (terms) built from variables and constants $\Z$ and $\U$ using action lattice operations, which are true in any action lattice from the given class under any valuation of variables. More precisely, the previous sentence defines the {\em in}equational theory, but in the presence of lattice operations it is equivalent to the equational one: $A \preceq B$ can be equivalently represented as $A \vee B = B$.

In a different terminology, equational theories of classes of action lattices are seen as algebraic logics. These logics are substructural, extending the multi\-pli\-ca\-ti\-ve-additive (``full'') Lambek calculus~\cite{Lambek1958,Kanazawa}, which is a non-commutative intuitionistic variant of Girard's linear logic~\cite{Girard1987}.

The equational theory of all action lattices is called {\em action logic} and denoted by $\mathbf{ACT}$. For the subclass of *-continuous action lattices, the equational theory is {\em infinitary action logic} $\mathbf{ACT}_\omega$, introduced by Buszkowski and Palka~\cite{Buszkowski2007,Palka2007,BuszkowskiPalka2008}. 

The interest to such a weak language---only (in)equations---is motivated by complexity considerations. Namely, for the next more expressible language, the language of Horn theories, the corresponding theory of the class of *-continuous action lattices is already $\Pi_1^1$-complete~\cite{Kozen2002}, that is, has a non-arithmetical complexity level. In contrast, $\mathbf{ACT}_\omega$ is $\Pi_1^0$-complete, as shown by Buszkowski and Palka~\cite{Buszkowski2007,Palka2007}. For the general case, $\mathbf{ACT}$ is $\Sigma_1^0$-complete~\cite{Kuznetsov2019LICS,Kuznetsov2021TOCL}, which is already the maximal possible complexity: iteration in action lattices in general allows a finite axiomatization, unlike the *-continuous situation, which requires infinitary mechanisms. 

Kleene algebras and their extensions are used in computer science for reasoning about program correctness. In particular, elements of an action lattice are intended to represent {\em types of actions} performed by a computing system (say, transitions in a finite automaton). Multiplication corresponds to composition of actions, Kleene star is iteration (perform an action several times, maybe zero). Residuals represent {\em conditional} types of actions. An action of type $a \imp b$, being preceded by an action of type $a$, gives an action of type $b$. Dually, $b \pmi a$ is the type of actions which require to be followed by an action of type $a$ to achieve $b$.

The monoid operation (multiplication) in action lattices is in general non-commutative, since so is, in general, composition of actions. However, in his original paper Pratt designates the subclass of commutative action algebras:
\begin{quotation}
\noindent
 ``A {\em commutative} action algebra is an action algebra satisfying $ab=ba$. Whereas action logic in general is neutral as to whether $ab$ combines $a$ and $b$ sequentially or concurrently, commutative action logic in effect commits to concurrency''.~\cite{Pratt1991}
\end{quotation}
Later on commutative action algebras (lattices) were not studied systematically. Concurrent computations are usually treated using a more flexible approach, using a specific parallel execution connective, $\parallel$, in the framework of concurrent Kleene algebras, CKA~\cite{HoareCKA2011}, and its extensions. In particular, the author is not aware of a study of equational theories (algebraic logics) for commutative action lattices.

Commutative versions of $\ACT$ and $\ACTomega$ are denoted by $\CommACT$ and $\CommACTomega$ respectively.
In this article, we prove undecidability and pinpoint the position in the arithmetical hierarchy for both $\CommACT$ and $\CommACTomega$. Namely, we prove that:
\begin{enumerate}
 \item $\CommACT$ is $\Sigma_1^0$-complete;
 \item $\CommACTomega$ is $\Pi_1^0$-complete.
\end{enumerate}
The second result was presented at the 3rd DaL{\'\i} Workshop and published in its proceedings~\cite{Kuznetsov2020DaLi}. The first result is new.

The rest of the article is organized as follows. We start with the *-continuous case. In Section~\ref{S:upper} we present an infinitary sequent calculus for $\CommACTomega$, prove cut elimination and the $\Pi_1^0$ upper bound. This construction basically copies Palka's~\cite{Palka2007} reasoning in the non-commutative case, for $\ACTomega$. Commutativity does not add anything significantly new here. 

In contrast, for proving $\Pi_1^0$-hardness (lower bound), which is performed in Section~\ref{S:ACTomega}, we could not have used Buszkowski's argument~\cite{Buszkowski2007}, since it uses a reduction from the totality problem for context-free grammars, which is instrinsically non-commutative. Instead, we use an encoding of 3-counter Minsky machines, which are commutative-friendly. The encoding of Minsky instructions and configurations is taken from the work of Lincoln et al.~\cite{LMSS}, with minor modifications. The principal difference from~\cite{LMSS}, however, is the usage of Kleene star to model {\em non-halting} behaviour of Minsky machines (while Lincoln et al. use the exponential modality of linear logic for modelling halting computations). 

In Section~\ref{S:ACT} we prove $\Sigma_1^0$-completeness for $\CommACT$ by encoding circular behaviour of Minsky machines and the technique of effective inseparability (Myhill's theorem). This argument is even more straightforward than the one from~\cite{Kuznetsov2019LICS,Kuznetsov2021TOCL}, since we do not need intermediate context-free grammars.

Section~\ref{S:conclusion} concludes the article by showing directions of further research in the area.

\section{Proof Theory and Upper Bounds}\label{S:upper}

\subsection{Sequent Calculi and Cut Elimination}

We present an infinitary sequent calculus for $\CommACTomega$, which is a commutative version of Palka's system for $\mathbf{ACT}_\omega$. Formulae of $\CommACTomega$ are built from a countable set of variables $\Var = \{ p,q,r, \ldots\}$ and constants $\Z$ and $\U$ using four binary connectives, $\imp$, $\mul$, $\vee$, and $\wedge$, and one unary connective, $^*$. (Due to commutativity, $B \pmi A$ is always equivalent to $A \imp B$, so we have only one residual here.) Sequents are expressions of the form $\Gamma \yields A$, where $\Gamma$ is a multiset of formulae (that is, the number of occurrences matters, while the order does not) and $A$ is a formula. In our notations, capital Greek letters denote multisets of formulae and capital Latin letters denote formulae.

Axioms and inference rules of $\CommACTomega$ are as follows:
$$
\infer[Id]{A \yields A}{}
\qquad
\infer[\Z L]{\Gamma, \Z \yields C}{}
\qquad
\infer[\U L]{\Gamma, \U \yields C}{\Gamma \yields C}
\qquad
\infer[\U R]{\yields \U}{}
$$
$$
\infer[\imp L]{\Gamma, \Pi, A \imp B \yields C}
{\Pi \yields A & \Gamma, B \yields C}
\qquad
\infer[\imp R]{\Pi \yields A \imp B}
{A, \Pi \yields B}
$$
$$
\infer[\mul L]{\Gamma, A \mul B \yields C}
{\Gamma, A, B \yields C} \qquad 
\infer[\mul R]{\Pi, \Delta \yields A \mul B}
{\Pi \yields A & \Delta \yields B}
$$
$$
\infer[\vee L]{\Gamma, A \vee B \yields C}
{\Gamma, A \yields C & \Gamma, B \yields C}
\qquad 
\infer[\vee R]{\Pi \yields A \vee B}
{\Pi \yields A}
\qquad 
\infer[\vee R]{\Pi \yields A \vee B}
{\Pi \yields B}
$$
$$
\infer[\wedge L]{\Gamma, A \wedge B \yields C}
{\Gamma, A \yields C}
\qquad
\infer[\wedge L]{\Gamma, A \wedge B \yields C}
{\Gamma, B \yields C}
\qquad
\infer[\wedge R]{\Pi \yields A \wedge B}
{\Pi \yields A & \Pi \yields B}
$$
$$
\infer[*L_\omega]{\Gamma, A^* \yields C}
{\bigl( \Gamma, A^n  \yields C \bigr)_{n=0}^{\infty}}
\qquad
\infer[*R_n,\ n \ge 0]{\Pi_1, \ldots, \Pi_n \yields A^*}
{\Pi_1 \yields A & \ldots & \Pi_n \yields A}
$$
$$
\infer[Cut]{\Gamma, \Pi \yields C}
{\Pi \yields A & \Gamma, A \yields C}
$$
The set of derivable sequents (theorems) is the smallest set which includes all instances of axioms and which is closed under inference rules. Thus, derivation trees in $\CommACTomega$ may have infinite branching (at instances of $*L$, which is an $\omega$-rule), but are required to be well-founded (infinite paths are forbidden).

Let us formulate several properties of $\CommACTomega$ and give proof sketches, following Palka~\cite{Palka2007}, but in the commutative setting. The proofs are essentially the same as Palka's ones; we give their sketches here in order to make this article logically self-contained.

The sequents of $\CommACTomega$ presented above enjoy a natural algebraic interpretation on commutative action lattices. Namely, given an action lattice $\Ac$, we intepret variables as arbitrary elements of $\Ac$, by a valuation function $v \colon \Var \to \Ac$, and then propagate this interpretation to formulae. Let us denote the interpretation of formula $A$ under valuation $v$ by $\bar{v}(A)$. A sequent of the form $A_1, \ldots, A_n \yields B$ ($n \ge 1$) is true under this interpretation if $\bar{v}(A_1) \mul \ldots \mul \bar{v}(A_n) \preceq \bar{v}(B)$ (due to commutativity of $\mul$, the order of $A_i$'s does not matter). For $n=0$, the sequent $\yields B$ is declared true if $\U \preceq \bar{v}(B)$. 
A soundness-and-completeness theorem holds:
\begin{theorem}
 A sequent is derivable in $\CommACTomega$ if and only if it is true in all commutative *-continuous action lattices under all valuations of variables.
\end{theorem}

\begin{proof}
 The ``only if'' part (soundness) is proved by (transfinite) induction on the structure of derivation. For the ``if'' part (completeness), we use the standard Lindenbaum -- Tarski canonical model construction.
\end{proof}

Thus, $\CommACTomega$ is indeed an axiomatization for the equational theory of commutative *-continuous action lattices.

In order to facilitate induction on derivation in the infinitary setting, we define the {\em depth} of a derivable sequent in the following way. For an ordinal $\alpha$, let us define the set $S_{\alpha}$ by transfinite recursion:
\begin{align*}
 & S_0 = \varnothing; \\
 & S_{\alpha + 1} = \{ \Gamma \yields A \mid \mbox{$\Gamma \yields A$ is derivable by one rule application from $S_{\alpha}$ } \}; \\
 & S_{\lambda} = \bigcup_{\alpha < \lambda} S_{\alpha} \mbox{ for $\lambda \in \mathrm{Lim}$.}
\end{align*}
(In particular, $S_1$ is the set of all axioms of $\CommACTomega$.)
For a derivable sequent $\Gamma \yields A$ let  $d(\Gamma \yields A) = \min \{ \alpha \mid (\Gamma \yields A) \in S_{\alpha} \}$ be its depth.

The complexity of a formula $A$ is defined as the total number of subformula occurrences in it.

\begin{theorem}
 The calculus $\CommACTomega$ enjoys cut elimination, that is, any derivable sequent can be derived without using Cut.
\end{theorem}

\begin{proof}
 First we eliminate one cut on the bottom of a derivation, that is, show that if $\Pi \yields A$ and $\Gamma, A \yields C$ are cut-free derivable, then so is $\Gamma, \Pi \yields C$. This is established by triple induction on the following parameters: (1) complexity of $A$; (2) depth of $\Pi \yields A$; (3) depth of $\Gamma, A \yields C$. See~\cite[Theorem~3.1]{Palka2007} for details.
 
 Next, let a sequent $\Gamma \yields B$ be derivable using cuts. Let $d(\Gamma \yields B)$ be its depth, counted for the calculus with $Cut$ as an official rule. Let us show that $\Gamma \yields B$ is cut-free derivable by induction on $\alpha = d(\Gamma \yields B)$. Notice that $\alpha$ is not a limit ordinal: otherwise, $(\Gamma \yields B) \in S_\beta$ for some $\beta < \alpha$. Also $\alpha \ne 0$. Thus, $\alpha = \beta + 1$. The sequent $\Gamma \yields B$ is immediately derivable, by one rule application, from a set of sequents from $S_{\beta}$, that is, of smaller depth. By the induction hypothesis, these sequents are cut-free derivable. Now consider the rule which was used to derive $\Gamma \yields B$. If it is not $Cut$, then $\Gamma \yields B$ is also cut-free derivable. If it is $Cut$, we apply the reasoning from the beginning of this proof and establish cut-free derivability of $\Gamma \yields B$. 
\end{proof}

The situation with $\CommACT$, the algebraic logic of all commutative action lattices, is different. This logic can be axiomatized, in the presence of $Cut$, by the following two axioms and an inductive rule for iteration:
\[
 \infer[*R_0]
 {\yields A^*}{}
 \qquad
 \infer[*R_{\mathrm{ind}}]
 {A, A^* \yields A^*}{}
 \qquad
 \infer[*L_{\mathrm{ind}}]
 {A^* \yields B}
 {\yields B & A, B \yields B}
\]
and the same rules for other connectives, as in $\CommACTomega$. This axiomatization of Kleene star exactly corresponds to its definition as $a^* = \min_\preceq \{ b \mid \U \preceq b \mbox { and } a \mul b \preceq b \}$. Thus,  soundness and completeness are established  by a standard Lindenbaum -- Tarski argument:
\begin{theorem}
 A sequent is derivable in $\CommACT$ if and only if it is true in all commutative action lattices under all valuations of variables.
\end{theorem}

This calculus for $\CommACT$, however, does not enjoy cut elimination, and there is no known cut-free formulation of $\CommACT$.

\subsection{The $\Pi_1^0$ Upper Bound for $\CommACTomega$}

For $\CommACT$, there is a trivial $\Sigma_1^0$ upper bound: any logic axiomatized by a calculus with finite proofs is recursively enumerable. In Section~\ref{S:ACT} we show that this complexity bound is exact, i.e., $\CommACT$ is $\Sigma_1^0$-complete.

For $\CommACTomega$, the situation is different. In general, such a calculus with an $\omega$-rule can be even $\Pi^1_1$-complete~\cite{KuznetsovSperanski2020}. In the non-commutative case, however, the complexity is much lower: $\ACTomega$ belongs to the $\Pi_1^0$ complexity class~\cite{Palka2007}. We show that  for $\CommACTomega$ the situation is the same.

In order to prove that $\CommACTomega$ belongs to the $\Pi_1^0$ complexity class, we use Palka's *-elimination technique. For each sequent, we define its {\em $n$-th approximation}. Informally, we replace each negative occurrence of $A^*$ with $A^{\le n} = \U \vee A \vee A^2 \vee \ldots \vee A^n$. The $n$-th approximation of a sequent $A_1, \ldots, A_m \yields B$ is defined as $N_n(A_1), \ldots, N_n(A_m) \yields P_n(B)$, where mappings $N_n$ and $P_n$ are defined by joint recursion:
\begin{align*}
 & N_n(\alpha) = P_n(\alpha) = \alpha,\ \alpha \in \Var \cup \{ \Z, \U \} && \\ 
 & N_n(A \imp B) = P_n(A) \imp N_n(B) && P_n(A \imp B) = N_n(A) \imp P_n(B) \\
 & N_n(A \mul B) = N_n(A) \mul N_n(B) && P_n(A \mul B) = P_n(A) \mul P_n(B) \\
 & N_n(A \vee B) = N_n(A) \vee N_n(B) && P_n(A \vee B) = P_n(A) \vee P_n(B) \\
 & N_n(A \wedge B) = N_n(A) \wedge N_n(B) && P_n(A \wedge B) = P_n(A) \wedge P_n(B) \\
 & \omit\rlap{$N_n(A^*) = \U \vee N_n(A) \vee (N_n(A))^2 \vee \ldots
 \vee (N_n(A))^n$} \\ & && P_n(A^*) = (P_n(A))^*
\end{align*}
(In Palka's notation, $N$ and $P$ are inverted.)

The *-elimination theorem, resembling Palka's~\cite{Palka2007} Theorem~5.1, is now formulated as follows:
\begin{theorem}\label{Th:starelim}
 A sequent is derivable in $\CommACTomega$ if and only if its $n$-th approximation is derivable in $\CommACTomega$ for any $n$.
\end{theorem}

\begin{proof}
 The ``only if'' part is easier. We establish by induction that $A \yields P_n(A)$ and $N_n(A) \yields A$ are derivable for any $A$: see~\cite[Lemma~4.3]{Palka2007} for $\mathbf{ACT}_\omega$; commutativity does not alter this part of the proof. Next, we apply $Cut$ several times:
 $$
 \infer{N_n(A_1), \ldots, N_n(A_m) \yields P_n(B)}
 {N_n(A_1) \yields A_1 & \ldots & N_n(A_m) \yields A_m
 & A_1, \ldots, A_m \yields B & B \yields P_n(B)}
 $$
 
 For the ``if'' part, a specific induction parameter is introduced. This parameter is called the {\em rank} of a formula and is represented by a sequence of natural numbers. These sequences are formally infinite, but include only zeroes starting from some point. For a sequent $\Gamma \yields A$ its rank $\rho(\Gamma \yields A)$ is the sequence $(c_0, c_1, c_2, \ldots)$, where $c_i$ is the number of subformulae of complexity $i$ in $\Gamma \yields A$.
 
 The order on ranks is anti-lexicographical: $(c_0, c_1, c_2, \ldots) \prec (c'_0, c'_1, c'_2, \ldots)$, if there exists a natural number $i$ such that $c_i < c'_i$ and for any $j>i$ we have $c_j = c'_j$. 
 In any rank $(c_0, c_1, c_2, \ldots)$ of a sequent there exists such a $k_0$ that $c_k = 0$ for all $k > k_0$ ($k_0$ is the maximal complexity of a subformula in $\Gamma \yields A$). Hence, any two ranks are comparable. Moreover, the order on ranks is well-founded. Thus, we can perform induction on ranks.
 
 The rules of $\CommACTomega$ (excluding $Cut$) enjoy the following property: each premise has a smaller rank than the conclusion. In particular, this holds for $*L$: despite $A$ is copied $n$ times, its complexity is smaller, than that of $A^*$. Thus, when going from conclusion to premise, we reduce some $c_i$ by one and increase $c_{i-1}$ (where $i$ is the complexity of $A^*$) and also some $c_j$'s with smaller indices. The rank gets reduced.
 
 Now we prove the ``if'' part of our theorem by contraposition. Suppose a sequent $\Pi \yields B$ is not derivable in $\CommACTomega$. We shall prove that for some $n$ the $n$-th approximation of this sequent is also not derivable. We proceed by induction on $\rho(\Pi \yields B)$. Consider two cases. 
 
 {\em Case 1:} one of the formulae in $\Pi$ is of the form $A^*$. Then $\Pi = \Pi', A^*$ and for some $m$ the sequent $\Pi', A^m \yields B$ is not derivable (otherwise $\Pi \yields B$ would be derivable by $*L$). 
 Since $\rho(\Pi', A^m \yields B) \prec \rho(\Pi', A^* \yields B)$, we can apply the induction hypothesis and conclude that for some $k$ the sequent $N_k(\Pi'), 
 (N_k(A))^m \yields P_k(B)$ is not derivable. Here $N_k(\Pi')$, for $\Pi' = C_1, \ldots, C_s$, is defined as $N_k(C_1)$, $\ldots$, $N_k(C_s)$.
 
 Now take $n = \max \{ m,k \}$. We claim that $N_n(\Pi'), N_n(A^*) \yields P_n(B)$ is not derivable. This is indeed the case, because otherwise we could derive the sequent $N_k(\Pi'), 
 (N_k(A))^m \yields P_k(B)$ using cut. The sequents used in cut are $N_k(C_j) \yields N_n(C_j)$, for each $C_j$ in $\Pi'$,  
 $(N_k(A))^m \yields N_n(A)^*$, and 
 $P_n(B) \yields P_k(B)$, which are derivable 
 (see~\cite[Lemma~4.4]{Palka2007}).

 {\em Case 2:} no formula of $\Pi$ is of the form $A^*$. Thus, our sequent cannot be derived using (immediately) the $*L$ rule. All other rules are finitary, and there is only a finite number of possible applications of these rules (for example, for $\imp L$ there is a finite number of possible splittings of the context to $\Gamma$ and $\Pi$). For each of these possible rule applications, at least one of its premises should be non-derivable (otherwise we derive the original sequent $\Pi \yields B$).
 
 The premises have smaller ranks than $\Pi \yields B$, so we can apply the induction hypothesis. This gives, for each premise, non-derivability of its $k$-th approximation for some $k$. Let $n$ be the maximum of these $k$'s. Increasing $k$ keeps each approximation non-derivable, and we get non-derivability of the $n$-th approximation of the original sequent. 
\end{proof}

The *-elimination technique yields the upper complexity bound:

\begin{theorem}
 The derivability problem in $\CommACTomega$ belongs to the $\Pi_1^0$ complexity class.
\end{theorem}

\begin{proof}
 By Theorem~\ref{Th:starelim}, derivability of a sequent is reduced to derivability of all its $n$-th approximations. Each $n$-th approximation, in its turn, is a sequent without negative occurrences of $^*$, that is, its derivation in $\CommACTomega$ is always finite (does not use $*L$). For such sequents, the derivability problem is decidable by exhausting proof search, since all rules, except $*L$, reduce the complexity of the sequent (when looking upwards). The ``$\forall n$'' quantifier yields $\Pi_1^0$. 
\end{proof}

 \section{$\Pi_1^0$-Hardness of $\CommACTomega$}\label{S:ACTomega}

  \subsection{Counter (Minsky) Machines}
 In our undecidability proofs, we encode counter machines, or Minsky machines~\cite{Minsky}, since in the commutative setting it is impossible to maintain order of letters and thus to encode Turing machines, semi-Thue systems, etc. 
 
  Let us recall some basics.
 A counter machine operates several {\em counters,} or registers, whose values are natural numbers. The machine itself is, at each point of operation, in a {\em state} $q$ taken from a finite set $Q$. Instructions of a counter machine are of the following forms, where $r$ is a register and $p, q, q_0, q_1$ are states:
 \begin{center}
 \begin{tabular}{l@{\quad}|@{\quad}l}
  $\INC(p, r, q)$ & being in state $p$, increase register $r$ by 1 \\ & and move to state $q$;\\[5pt]
  $\JZDEC(p, r, q_0, q_1)$ & being in state $p$, check
  whether the value of $r$ is 0: \\ &
  if yes, move to state $q_0$, \\ &
  if no, decrease $r$ by 1 and move to state $q_1$.
 \end{tabular}
\end{center}

In what follows, we consider only deterministic counter machines, that is, for each state $p$ there exists no more than one instruction with this $p$ as the first parameter. 
Moreover, there is a unique state for which there is no such instruction, and this state is called the final one and denoted by $q_F$. The machine {\em halts} once it reaches $q_F$.

Counter machines are used for computing partial functions on natural numbers. One fixed register, denoted by $\ra$, is used for input/output: the machine starts at the initial state $q_S$ with its input data (a natural number) put into $\ra$; all other registers are assigned to 0. If the machine halts, then the resulting value is located in $\ra$. We may suppose that other registers hold 0; otherwise we can add extra states and instructions to perform ``garbage collection.'' If the machine does not halt (runs forever), the function on the given input is undefined.

A {\em configuration} of a counter machine is a tuple of the form $\langle q, c_1, \ldots, c_n \rangle$, where $q \in Q$, $n$ is the number of registers, and $c_1, \ldots, c_n$ are natural numbers (values of registers). The starting configuration, on input $x$, is $\langle q_S, x, 0, \ldots, 0 \rangle$. 

We restrict ourselves to 3-counter machines, with only three registers: $\ra$, $\rb$, and $\rc$, as three registers are sufficient for Turing completeness. Namely, any computable partial function on natural numbers can be computed on a 3-counter machine as defined above.

An accurate translation from Turing machines to 3-counter ones can be found in Schroeppel's memo~\cite{Schroeppel}. 
Notice that 2-counter machines are also Turing-complete, but in a specific sense: a natural number $n$ should be submitted as an input not as it is, but as $2^n$, and the same for output~\cite{Minsky}; the function $n \mapsto 2^n$ itself is not computable on 2-counter machines~\cite{Schroeppel}. To avoid this inconvenience, we use 3-counter machines.
 
\begin{prop}\label{Pr:Turing}
 A partial function $f \colon \NN \to \NN$ is computable if and only if $f$ is computed by a 3-counter machine.~{\rm\cite{Schroeppel}}
\end{prop}

It will be convenient for us to use the definition of {\em recursively enumerable} (r.e., or $\Sigma_1^0$) sets as domains of computable functions: $D_f = \{ x \mid \mbox{$f(x)$ is defined} \}$ or, in view of Proposition~\ref{Pr:Turing}, 
$D_\Mf = \{ x \mid \mbox{$\Mf$ halts on input $x$} \}$. Among r.e. sets, there exist $\Sigma_1^0$-complete ones. Thus, the general halting problem for 3-counter machines is $\Sigma_1^0$-complete. The dual {\em non-halting} problem is $\Pi_1^0$-complete; moreover, there exists a concrete $\Mf$ such that $\overline{D}_\Mf = \{ x \mid \mbox{$\Mf$ does not halt on $x$}\}$ is $\Pi_1^0$-complete.

\subsection{Encoding Minsky Instructions}

We prove $\Pi_1^0$-hardness of $\CommACTomega$ by reducing the non-halting problem for deterministic 3-counter Minsky machines to derivability in $\CommACTomega$. Our approach is in a sense dual to the undecidability proof for commutative propositional linear logic by Lincoln et al.~\cite{LMSS}. They use the exponential modality, $!A$, which is expanded to $A^n$ for {\em some} $n$, using the contraction rule. The formula $A$ being an encoding of the instruction set of a Minsky machine, this construction represents termination of Minsky computation after $n$ steps. Dually, we use $A^*$, which is expanded using the $\omega$-rule, $*L$, to an infinite series of sequents with $A^n$ for {\em any} $n$. This corresponds to an infinite run of the Minsky machine: it can perform arbitrarily many steps.

Notice that, as in~\cite{LMSS}, we essentially use commutativity. It is needed to deliver the instruction to the correct place in the formula encoding the machine configuration. In the non-commutative setting, this is a separate issue, and Buszkowski's $\Pi_1^0$-hardness proof for $\mathbf{ACT}_\omega$~\cite{Buszkowski2007}  uses an indirect reduction from non-halting of Turing machines, via totality for context-free grammars.

Let $\Mf$ be a deterministic 3-counter machine. In $\CommACTomega$, configurations of $\Mf$ are encoded as follows. Let the set of variables include the set of states $Q$ of $\Mf$, and additionally three variables $\ra$, $\rb$, and $\rc$ for counters. Configuration $\langle q,a,b,c \rangle$ is encoded as follows:
\[
 q, \underbrace{\ra, \ldots, \ra}_{\text{$a$ times}}, 
 \underbrace{\rb, \ldots, \rb}_{\text{$b$ times}}, 
 \underbrace{\rc, \ldots, \rc}_{\text{$c$ times}}.
\]
This encoding will appear in antecedents of $\CommACTomega$ sequents, thus, it is considered as a multiset. This keeps the numbers of $\ra$'s, $\rb$'s, and $\rc$'s, which is crucial for representing Minsky configurations.

Each instruction $I$ of $\Mf$ is encoded by a specific formula $A_I$. For $\INC$, the encoding is straightforward:
$$
A_{\INC(p,r,q)} = p \imp (q \mul r).
$$
For $\JZDEC$, the encoding is more involved. We introduce two extra variables, $z_{\ra}$ and
$z_\rb$, and encode $\JZDEC(p,r,q_0,q_1)$ by the following formula:
$$
A_{\JZDEC(p,r,q_0,q_1)} = 
((p \mul r) \imp q_1) \wedge (p \imp (q_0 \vee
 z_r)).
$$

Moreover, we introduce three extra formulae,
$N_{\ra} = z_{\ra} \imp z_{\ra}$, 
$N_{\rb} = z_{\rb} \imp z_{\rb}$, and
$N_{\rc} = z_{\rc} \imp z_{\rc}$.

Let us explain the informal idea behind this encoding. Suppose that we wish to model $n$ steps of execution. In our derivations, formulae of the form $A_I$ are going to appear in left-hand sides of sequents
(along with the code of the configuration), instantiated using Kleene star (we consider the derivation of the $n$-th premise of $*L$). For $\INC$, when the formula $A_{\INC(p,r,q)}$ gets introduced by $\imp L$, we replace $p$ with $q \mul r$ (looking from bottom to top). This corresponds to changing the state from $p$ to $q$ and increasing register $r$.

For $\JZDEC$, we use additive connectives, $\wedge$ and $\vee$. Being in the negative position (in the left-hand side of the sequent), $\wedge$ implements {\em choice} and $\vee$ implements {\em branching} (parallel computations). In $\JZDEC$, the choice is as follows. If there is at least one copy of variable $r$ (i.e., the value of register $r$ is not zero), we can choose $(p \mul r) \imp q_1$ which changes the state from $p$ to $q_1$ and decreases $r$. We could also choose $p \imp (q_0 \vee z_r)$, for the zero case. This operation continues the main execution thread by changing to state $q_0$, but also forks a new thread with a ``state'' $z_r$. This new thread is designed to check whether $r$ is actually zero. Since the thread was forked in the middle of the execution, say, after $k$ steps, it still has to perform $(n-k)$ steps of execution. They get replaced by dummy instructions, encoded by $N_r = z_r \imp z_r$.

The set of instructions (including ``dummies'') is encoded by the formula
$$
E = N_\ra \wedge N_\rb \wedge N_\rc \wedge \bigwedge_I A_I,
$$
which is going to be copied using Kleene star.

The key feature of our encoding is the right-hand side of the sequent, which is going to be 
$$ 
D = \bigl( \ra^* \mul \rb^* \mul \rc^* \mul \bigvee_{q \in Q} q \bigr) \vee (\rb^* \mul \rc^* \mul z_\ra) \vee (\ra^* \mul \rc^* \mul z_\rb) \vee (\ra^* \mul \rb^* \mul z_\rc).
$$
This formula represents constraints on the configuration after performing $n$ steps of computation. For the main execution thread, it just says that it should reach a correctly encoded configuration of the form $\langle q, a, b,c \rangle$, $q \in Q$, $a,b,c \in \NN$. For zero-checking thread, with ``state'' $z_r$, $D$ enforces the value of register $r$ to be zero.

In the next subsection, we formulate and prove a theorem which establishes a correspondence between Minsky computations and derivations of specific sequents in $\CommACTomega$.

\subsection{Computations and Derivations}

\begin{theorem}\label{Th:omega}
 Minsky machine $\Mf$ runs forever on input $x$ if and only if the sequent \[E^*, q_S, \ra^x \yields D \eqno{(*)}\] is derivable in $\CommACTomega$. Therefore, $\CommACTomega$ is $\Pi_1^0$-hard.
\end{theorem}

The proof of this theorem is based on the following lemma.

\begin{lemma}\label{Lm:omega}
 Minsky machine $\Mf$ can perform $k$ steps of execution starting from configuration $\langle p, a,b,c \rangle$ if and only if the sequent $E^k,  p, \ra^a, \rb^b, \rc^c \yields D$. 
\end{lemma}

Indeed, let $p = q_S$, $a = x$, and $b=c=0$. Then $E^*, q_S, \ra^x \yields D$ is derivable from $\bigl( E^n, q_S, \ra^x \yields D \bigr)_{n=0}^\infty$ by $*L$, and the opposite implication is by cut with $E^n \yields E^*$. Thus, derivability of $(*)$ is equivalent to the fact that $\Mf$ can perform arbitrarily many steps starting from $\langle q_S, x,0,0 \rangle$. Since $\Mf$ is deterministic, this is equivalent to infinite run.

\begin{proof}[Proof of Lemma~\ref{Lm:omega}]
 The {\bf ``only if''} part, from computation to derivation, is easier. We proceed by induction on $k$. In the base case, $k=0$, we derive the necessary sequent, $p, \ra^a, \rb^b, \rc^c \vdash D$ by $\vee R$ (twice) from $p, \ra^a, \rb^b, \rc^c \vdash \ra^* \cdot \rb^* \cdot \rc^* \cdot \bigvee_{q\in Q} q$. The latter is derived using $*R$, $\vee R$, and $\cdot R$. 
 
 For the induction step, consider the first $\Mf$'s instruction executed. If it is $\INC(p, \ra, q)$, we 
 perform the following derivation:
 \[\small
 \infer=[\wedge L]
 {E^k, p, {\ra}^a, {\rb}^b, \rc^c \yields D}
 {\infer[\imp L\text{ ($A_{\INC(p,\ra,q)} = p \imp (q \mul \ra)$)}]{E^{k-1}, A_{\INC(p, \ra, q)}, p,
 {\ra}^a, {\rb}^b, \rc^c \yields D}
 {p \yields p  &
 \infer[\mul L]{E^{k-1}, q \mul \ra, {\ra}^a, {\rb}^b, \rc^c \yields D}
 {E^{k-1}, q, {\ra}^{a+1}, {\rb}^b, \rc^c \yields D}}}
 \]
 Here and further double horizontal line means several applications of a rule.
 
  The topmost sequent $E^{k-1}, q, \ra^{a+1}, \rb^b, \rc^c \yields D$ is derivable by inductive hypothesis, since $\Mf$ can perform $k-1$ execution steps starting from the next configuration $\langle q, a+1, b, c \rangle$. 
 Instructions $\INC(p,\rb,q)$ and $\INC(p,\rc,q)$ are considered similarly.
 
 For $\JZDEC(p,\ra, q_0,q_1)$, we consider two cases. If $a \ne 0$, then the derivation is similar to the one for $\INC$:
 \[\small
 \infer=[\wedge L]
 {E^k, p, \ra^a, \rb^b, \rc^c \yields D}
 {\infer[\wedge L]{E^{k-1}, A_{\JZDEC(p,\ra,q_0,q_1)},
 p, \ra^a, \rb^b, \rc^c \yields D}
 {\infer[\imp L]{E^{k-1}, (p \mul \ra) \imp q_1,
 p, \ra^a, \rb^b, \rc^c \yields D}
 {\infer[\mul R]{p, \ra \yields p \mul \ra}
 {p \yields p & \ra \yields \ra} & 
 E^{k-1}, q_1, \ra^{a-1}, \rb^b, \rc^c \yields D}}}
 \]
 Here $E^{k-1}, q_1, \ra^{a-1}, \rb^b, \rc^c \yields D$ is derivable by the induction hypothesis.
 
 The interesting part is the zero test. Let $a = 0$ and perform the following derivation:
 \[\small
 \infer=[\wedge L]
 {E^k, p, \rb^b, \rc^c \yields D}
 {\infer[\wedge L]{
 E^{k-1}, A_{\JZDEC(p,\ra,q_0,q_1)}, p, \rb^b, \rc^c 
 \yields D}
 {\infer[\imp L]{
 E^{k-1}, p \imp (q_0 \vee z_\ra), p, 
 \rb^b, \rc^c \yields D}
 {p \yields p & \infer[\vee L]
 {E^{k-1}, q_0 \vee z_\ra, \rb^b, \rc^c,  \yields D}
 {E^{k-1}, q_0, \rb^b, \rc^c \yields D & 
 E^{k-1}, z_\ra, \rb^b, \rc^c \yields D}
 }
 }
 }
 \]
 On the left branch, we have $E^{k-1}, q_0, \rb^b, \rc^c \yields D$, which is derivable by the induction hypothesis: $\langle q_0, 0, b, c \rangle$ is the successor for $\langle p, 0, b ,c \rangle$ after applying $\JZDEC(p,\ra,q_0,q_1)$.
 
 The sequent on the right branch, $E^{k-1}, z_\ra, \rb^b, \rc^c \yields D$, can be derived using $\wedge L$ and $\vee R$ from $(z_\ra \imp z_\ra)^{k-1}, z_\ra, \rb^b, \rc^c \yields \rb^* \mul \rc^* \mul z_\ra$. Indeed, $E$ is a conjunction which includes $N_\ra = z_\ra \imp z_\ra$, and $D$ is a disjunction which includes $\rb^* \mul \rc^* \mul z_\ra$. The latter sequent, $(z_\ra \imp z_\ra)^{k-1},  z_\ra, \rb^b, \rc^c \yields \rb^* \mul \rc^* \mul z_\ra$, is derivable.
 
 The cases of $\JZDEC(p,\rb,q_0,q_1)$ and $\JZDEC(p, \rc, q_0, q_1)$ are similar.
 
 \vskip 5pt
 
 For the {\bf ``if''} part we analyze the cut-free derivation of~$E^k, p, \ra^a, \rb^b, \rc^c \vdash D$. It is important to notice that this derivation does not necessarily directly represent a $k$-step workflow of $\Mf$ as shown above.
 
 \begin{example}
  Let $\Mf$ include the following instructions: $\INC(p,\ra,q)$ and $\JZDEC(q,\ra,p,p)$, and consider a 4-step execution of $\Mf$ starting from $\langle p,0,0,0 \rangle$. Such an execution is indeed possible, and can be represented by the following ``canonical'' derivation:
  \[\small
   \infer=[\wedge L]
   {E^4, p \yields D}
   {\infer[\mul L, \imp L]
   {E^3, p \imp (q \mul \ra), p \yields D}
   {p \yields p & 
   \infer=[\wedge L]{E^3, q, \ra \yields D}
   {\infer[\imp L]{E^2, (q \mul \ra) \imp p, q, \ra \yields D}
   {q, \ra \yields q \mul \ra & \infer{E^2, p \yields D}{\infer{\raisebox{4pt}{\vdots}}{p \yields D}}}}
   }}   
  \]

However, there is also an alternative derivation:
\[\small
 \infer=[\wedge L]
   {E^4, p \yields D}
   {\infer[\imp L]{E^3, p \imp (q \mul \ra), p \yields D}{
   \infer=[\wedge L]{E^2, p \yields p}
   {\infer[\mul L, \imp L]{E, p \imp (q \mul \ra), p \yields p}{p \yields p & 
   \infer=[\wedge L]{E, q, \ra \yields p}{\infer[\imp L]{(q \mul \ra) \imp p, \ra, q \yields p}{q, \ra\yields q \mul \ra & p \yields p}}}} & 
   \infer[\mul L]{E, q \mul \ra \yields D}
   {\infer=[\wedge L]{E, q, \ra \yields D}
   {\infer[\imp L]{(q \mul \ra) \imp p, q, \ra \yields D}{q, \ra \yields q \mul \ra & p \yields D}}}}}
\]

In this derivation, there is a ``subroutine'' (the left subtree) which moves from $\langle p,0,0,0 \rangle$ to $\langle p,0,0,0 \rangle$ in 2 steps.
\end{example}

In general, such ``subroutines'' could be represented by subderivations for sequents of the form $E^m, q, \ra^{a'}, \rb^{b'}, \rc^{c'} \yields p$ (while in the ``canonical'' derivation they are all trivialized to $p \yields p$). This corresponds to $\langle q,a',b',c' \rangle \to \langle p,0,0,0 \rangle$ in $m$ steps. The crucial observation, however, is that in such ``subroutines'' $\JZDEC$ cannot branch to the zero ($r=0$) case. The reason is that in the subtree there is no $D$ which supports the usage of $z_r$. Therefore, such a ``subroutine'' also validates the transition 
$\langle q,a'+a,b'+b,c'+c \rangle \to \langle p,a,b,c \rangle$ for arbitrary $a,b,c$. This allows connecting the ``subroutine'' to the main execution workflow.

The idea described above is formalized in the usual boring way, proving several statements by joint induction. Let $\widetilde{E}_i$ denote any formula in the conjunction $E$ or a conjunction of such formulae (in particular, $\widetilde{E}_i$ could be $E$ itself). For convenience, let $R = \{\ra,\rb,\rc\}$, $Z = \{ z_\ra, z_\rb, r_\rc \}$, and $Z_{\bar{r}} = Z - \{ z_r \}$ (e.g., $Z_{\bar{\rb}} = \{ z_\ra, z_\rc \}$).

\begin{enumerate}
 \item\label{St:bad} Sequents of the form $\widetilde{E}_1, \ldots, \widetilde{E}_k, \ra^a, \rb^b, \rc^c \yields t$, where $t \in Q \cup Z$, are never derivable, neither are sequents of the form 
 $\widetilde{E}_1, \ldots, \widetilde{E}_k, \ra^a, \rb^b, \rc^c \yields t \cdot r$, where $r \in R$.
 \item\label{St:zsub} Sequents of the form $\widetilde{E}_1, \ldots, \widetilde{E}_k,z_r, \ra^a, \rb^b, \rc^c \yields t$, where $r \in R$ and $t \in Q \cup Z_{\bar{r}}$, are never derivable, neither are sequents of the form 
 $\widetilde{E}_1, \ldots, \widetilde{E}_k, z_r, \ra^a, \rb^b, \rc^c \yields t \cdot r'$, where $r, r' \in R$ and $t \in Q \cup Z_{\bar{r}}$. 
 \item\label{St:z} If $\widetilde{E}_1, \ldots, \widetilde{E}_k, z_\ra, \ra^a, \rb^b, \rc^c \yields D$ is derivable, then $a = 0$. Similarly for $\rb$ and $\rc$.
 \item\label{St:sub} If $\widetilde{E}_1, \ldots, \widetilde{E}_k, q, \ra^{a'}, \rb^{b'}, \rc^{c'} \yields p$ is derivable, where $p,q \in Q$, then $\Mf$ can move from $\langle q,a'+a,b'+b,c'+c \rangle$ to $\langle p,a,b,c \rangle$ in $k$ steps for any $a,b,c$.
 \item\label{St:subr} If $\widetilde{E}_1, \ldots, \widetilde{E}_k, q, \ra^{a'}, \rb^{b'}, \rc^{c'} \yields p \mul \ra$, where $p,q \in Q$, is derivable, then $\Mf$ can move from $\langle q, a' + a, b' + b, c' + c \rangle$ to $\langle p, a+1, b,c \rangle$ in $k$ steps for any $a,b,c$. Similarly for $\ra$, $\rb$, $\rc$. 
 \item\label{St:main} If $\widetilde{E}_1, \ldots, \widetilde{E}_k, p, \ra^a, \rb^b, \rc^c \yields D$ is derivable ($p \in Q$), then $\Mf$ can perform $k$ steps, starting from $\langle p,a,b,c \rangle$.
\end{enumerate}

Statement~\ref{St:main} with $\widetilde{E}_1 = \ldots = \widetilde{E}_k = E$ yields our goal (the ``if'' part of Lemma~\ref{Lm:omega}).

Rules $\vee L$ and $\mul L$ are invertible (this can be established using cut), so we can suppose that in our derivations they are always applied immediately.

Next, we reorganize the derivation so that no right rule ($\vee R$, $\mul R$, or $*R$) appears below a left rule ($\wedge L$, $\vee L$, $\imp L$, $\mul L$). Such a reorganization is possible since $\imp R$ is never applied (there are no formulae of the form $F \imp G$ in succedents). For example, $\mul R$ and $\vee L$ are exchanged in the following way:
$$\small
 \infer[\mul R]{\Pi', E \vee F, \Pi'', \Delta \yields A \mul B}
 {\infer[\vee L]{\Pi', E \vee F, \Pi'' \yields A}
 {\Pi', E, \Pi'' \yields A & 
 \Pi', F, \Pi'' \yields A} & 
 \Delta \yields B}
 $$
 transforms into
 $$\small
 \infer[\vee L]{\Pi', E \vee F, \Pi'', \Delta \yields A \mul B}
 {\infer[\mul R]{\Pi', E, \Pi'', \Delta \yields A \mul B}
 {\Pi', E, \Pi'' \yields A & \Delta \yields B} &
 \infer[\mul R]{\Pi', F, \Pi'', \Delta \yields A \cdot B}{\Pi', F, \Pi'' \yields A & \Delta \yields B}}
 $$
 Transformations in other are similar.

 Now let us prove our statements by joint induction on $k$. The base cases ($k=0$) are considered as follows. For statements~\ref{St:bad} and~\ref{St:zsub}, we have a ``lonely'' $t$ in the succedent, which could not be matched with another $t$ to form an axiom. Thus, the sequents are not derivable. In statement~\ref{St:z}, when deriving $z_\ra, \ra^a, \rb^b, \rc^c \yields D$, we have to choose $\rb^* \mul \rc^* \mul z_\ra$ from $D$ (otherwise $z_\ra$ does not have a match). Therefore, $a = 0$, since there are no occurrences of $\ra$ in the succedent. In statement~\ref{St:sub}, the sequent should be of the form $p \yields p$, that is, $q = p$ and $a' = b' = c' = 0$. The 0-step move from $\langle p,a,b,c \rangle$ to $\langle p,a,b,c \rangle$ is trivial. 
 Similarly, for statement~\ref{St:subr}, we have exactly $p, \ra \yields p \mul \ra$, thus, $q = p$, $a' = 1$, and $b' = c' = 0'$, and a 0-step move from $\langle p,a+1,b,c \rangle$ to $\langle p,a+1,b,c \rangle$.
 Finally, the base case for statement~\ref{St:main} is obvious, since performing 0 steps is always possible.
 
Now let  $k \ne 0$.
If the lowermost rule applied in our derivation is $\wedge L$, then it just changes one of the $\widetilde{E}_i$'s to a formula of the same form (formally, here we use a nested induction on derivation height). Thus, the interesting case is $\imp L$, when one of the $\widetilde{E}_i$'s is of the form $F \imp G$, and it gets decomposed. (Notice that from $A_{\JZDEC(p,r,q_0,q)}$ we have already taken, or ``chosen,'' only one conjunct.)

{\em Statement~\ref{St:bad}.} We have $\widetilde{E}_i = F \imp G$, and the left premise of $\imp L$ is again of the form $\widetilde{E}_1, \ldots, \widetilde{E}_{k'}, \ra^{a'}, \rb^{b'}, \rc^{c'} \yields F$, where $k' < k$ and $F$ is either $t'$ or $t' \cdot r'$, $t' \in Q \cup Z$, $r' \in R$. The latter is due to the form of conjuncts in $E$. Since $k' < k$, we can apply the induction hypothesis and conclude that the left premise is not derivable.

{\em Statement~\ref{St:zsub}.} The occurrence of $z_r$ should go to the left premise of $\imp L$, otherwise we face contradiction with statement~\ref{St:bad}. Consider two cases. If $\widetilde{E}_i = F \imp G = z_r \imp z_r$ (with the same $r$), then the right premise of $\imp L$ is again of the same form, as the goal sequent, but with a smaller $k$, and we proceed by induction. Otherwise, $F$ is of the form $t'$ or $t' \cdot r''$, where $t' \in Q \cup Z_{\bar{r}}$ and $r'' \in R$. In this case we apply the induction hypothesis to the left premise.

{\em Statement~\ref{St:z}.} Again, by statement~\ref{St:bad} $z_\ra$ should go to the left premise. If $F \ne z_{\ra}$, then derivability of the left premise contradicts statement~\ref{St:zsub}. Thus, $\widetilde{E}_i = z_{\ra} \imp z_{\ra}$, and we apply the induction hypothesis to the right premise.

{\em Statement~\ref{St:sub}.} Again, $q$ should go to the left premise. If $\widetilde{E}_i = z_r \imp z_r$, then the right premise is of the form $\widetilde{E}_1, \ldots, \widetilde{E}_{k'}, z_r, \ra^{a''}, \rb^{b''}, \rc^{c''} \yields p$ and could not be derivable by statement~\ref{St:zsub}. Thus, three cases remain possible; for simplicity, let $r = \ra$, the cases of $r = \rb$ and $r = \rc$ are handled in the same way.
\begin{itemize}
 \item $\widetilde{E}_i = A_{\INC(p',r,q')} = p' \imp (q' \mul \ra)$. We have the following application of $\imp L$ and an immediate application of $\mul L$:
 \[
 \small
 \infer[\imp L]
 {\widetilde{E}_1, \ldots, \widetilde{E}_k, q, \ra^{a'}, \rb^{b'}, \rc^{c'} \yields p}
 {\widetilde{E}_1, \ldots, \widetilde{E}_{i-1}, q, \ra^{a_1}, \rb^{b_1}, \rc^{c_1} \yields p' & 
 \infer[\mul L]{\widetilde{E}_{i+1}, \ldots, \widetilde{E}_{k}, q' \mul \ra, \ra^{a_2}, \rb^{b_2}, \rc^{c_2} \yields p}
 {\widetilde{E}_{i+1}, \ldots, \widetilde{E}_{k}, q', \ra^{a_2+1}, \rb^{b_2}, \rc^{c_2} \yields p}}
 \]
Here $a_1 + a_2 = a'$, $b_1 + b_2 = b'$, $c_1 + c_2 = c'$. By induction hypothesis, $\Mf$ can move from $\langle q, a_1 + a_2 + a, b_1 + b_2 + b, c_1 + c_2 + c \rangle$ to $\langle p', a_2 + a, b_2 + b, c_2 + c \rangle$ in $i-1$ steps.  Next, we apply $\INC(p',r,q')$ and change the configuration to $\langle q', a_2+1+a, b_2+b,c_2+c \rangle$. Finally, we move to $\langle p, a,b,c \rangle$ in $k-i$ steps again by induction hypothesis. The total number of steps is $(i-1)+1+(k-i) =k$.
\item $\widetilde{E}_i = (p' \mul \ra) \imp q'$, the first part of $A_{\JZDEC(p',r,q_0,q')}$. 
\[\small
 \infer[\imp L]
 {\widetilde{E}_1, \ldots, \widetilde{E}_k, q, \ra^{a'}, \rb^{b'}, \rc^{c'} \yields p}
 {\widetilde{E}_1, \ldots, \widetilde{E}_{i-1}, q, \ra^{a_1}, \rb^{b_1}, \rc^{c_1} \yields p' \mul \ra & 
 \widetilde{E}_{i+1}, \ldots, \widetilde{E}_k, q', \ra^{a_2}, \rb^{b_2}, \rc^{c_2} \yields p}
\]
Applying statement~\ref{St:subr} to the left premise yields that $\Mf$ can move from $\langle q, a_1 + a_2 + a, b_1 + b_2 + b, c_1 + c_2 + c \rangle$ to $\langle p', a_2 + a + 1, b_2 + b, c_2 + c \rangle$ in $i-1$ steps. Next, since $a_2 + a + 1 > 0$, applying $\JZDEC(p',r,q_0,q')$ changes the configuration to $\langle q', a_2 + a_1, b_2 + b, c_2 +c \rangle$. Finally, by induction hypothesis (statement~\ref{St:subr}) applied to the left premise, we reach $\langle p,a,b,c \rangle$ in $k-i$ steps.
\item $\widetilde{E}_i = p' \imp (q_0 \vee z_\ra)$, the second  part of $\JZDEC(p',r,q_0,q')$. Taking $r = \ra$, we get the following application of $\imp L$, preceded by an immediate application of $\vee L$:
\[\small
 \infer[\imp L]
 {\widetilde{E}_1, \ldots, \widetilde{E}_k, q, \ra^{a'}, \rb^{b'}, \rc^{c'} \yields p}
 {\widetilde{E}_1, \ldots, \widetilde{E}_{i-1}, q,
 \ra^{a_1}, \rb^{b_1}, \rc^{c_1} \yields p' & 
 \infer[\vee L]{
 \widetilde{E}_{i+1}, \ldots, \widetilde{E}_k, q' \vee z_\ra, \ra^{a_2}, \rb^{b_2}, \rc^{c_2} \yields p}
 {\widetilde{E}_{i+1}, \ldots, \widetilde{E}_k, z_\ra, \ra^{a_2}, \rb^{b_2}, \rc^{c_2} \yields p & \ldots, q', \ldots \yields p}
 }
\]
One of the premises of $\vee L$, namely, $\widetilde{E}_{i+1}, \ldots, \widetilde{E}_k, z_\ra, \ra^{a_2}, \rb^{b_2}, \rc^{c_2} \yields p$, is not derivable by statement~\ref{St:z} (induction hypothesis).
\end{itemize}

{\em Statement~\ref{St:subr}} is established in the same way as statement~\ref{St:sub}, with a routine change of $p$ to $p \cdot \ra$ and $a$ to $a+1$.

{\em Statement~\ref{St:main}.} The difference from the previous two statements is that here, on the main thread of derivation, the second case of $\JZDEC$ can realized.

As for previous statements, $q$ should go to the left premise and $\widetilde{E}_i$ could not be $z_r \imp z_r$ (due to statements~\ref{St:bad} and~\ref{St:zsub}).
Let $r = \ra$ and consider the three possible cases.

\begin{itemize}
 \item $\widetilde{E}_i = A_{\INC(p',r,q')} = p' \imp (q' \mul \ra)$. 
 \[\small
  \infer[\imp L]
  {\widetilde{E}_1, \ldots, \widetilde{E}_k, p , \ra^a, \rb^b, \rc^c\yields D}
  {\widetilde{E}_1, \ldots, \widetilde{E}_{i-1}, q, \ra^{a_1}, \rb^{b_1}, \rc^{c_1} \yields p' & 
  \infer[\mul L]
  {\widetilde{E}_{i+1}, \ldots, \widetilde{E}_k, q' \cdot \ra, \ra^{a_2}, \rb^{b_2}, \rc^{c_2} \yields D}
  {\widetilde{E}_{i+1}, \ldots, \widetilde{E}_k, q', \ra^{a_2+1}, \rb^{b_2}, \rc^{c_2} \yields D}}
 \] By induction hypothesis, statement~\ref{St:sub}, $\Mf$ can move from $\langle q, a, b, c \rangle$, which is $\langle q, a_1 + a_2, b_1 + b_2, c_1 + c_2 \rangle$, to $\langle p', a_2, b_2, c_2 \rangle$ in $i-1$ steps. Next, applying $\INC(p',r,q')$ yields $\langle q', a_2+1, b_2, c_2 \rangle$.
 Finally, by induction hypothesis, statement~\ref{St:main}, applied to the right premise, $\Mf$ can perform $k-i$ more steps. The total number of steps performed starting from $\langle q,a,b,c \rangle$  equals $k$.
 
 \item $\widetilde{E}_i = (p' \mul \ra) \imp q'$, the first part of $A_{\JZDEC(p',\ra,q_0,q')}$. 
 \[
  \small
  \infer[\imp L]
  {\widetilde{E}_1, \ldots, \widetilde{E}_k, p, \ra^a, \rb^b, \rc^c \yields D}
  {\widetilde{E}_1, \ldots, \widetilde{E}_{i-1}, q, \ra^{a_1}, \rb^{b_1}, \rc^{c_1} \yields p' \mul \ra
   &
   \widetilde{E}_{i+1}, \ldots, \widetilde{E}_k, q', \ra^{a_2}, \rb^{b_1}, \rc^{c_1} \yields D}
 \]
By induction hypothesis, statement~\ref{St:subr}, $\Mf$ can move from $\langle q, a, b, c \rangle = \langle q, a_1 + a_2, b_1 + b_2, c_1 + c_2 \rangle$ to $\langle p', a_2 + 1, b_2, c_2 \rangle$ in $i-1$ steps. Since $a_2 + 1 > 0$, applying $\JZDEC(p',r,q_0,q')$ changes $\langle p', a_2 + 1, b_2, c_2 \rangle$ to $\langle q', a_2, b_2, c_2 \rangle$. Finally, applying statement~\ref{St:main} (induction hypothesis), we perform the remaining $k-i$ steps. 

\item $\widetilde{E}_i = p' \imp (q_0 \vee z_\ra)$, the second part of $A_{\JZDEC(p',\ra,q_0,q')}$.  The goal sequent $\widetilde{E}_1, \ldots, \widetilde{E}_k, p, \ra^a, \rb^b, \rc^c \yields D$ is derived, using $\imp L$, from sequents $\widetilde{E}_1, \ldots, \widetilde{E}_{i-1}, q,
 \ra^{a_1}, \rb^{b_1}, \rc^{c_1} \yields p'$ and $\widetilde{E}_{i+1}, \ldots, \widetilde{E}_k, q_0 \vee z_\ra, \ra^{a_2},  \rb^{b_2}, \rc^{c_2} \yields D$, where the latter is derived by an immediate application of $\vee L$:
 \[\small
  \infer[\vee L]
 {\widetilde{E}_{i+1}, \ldots, \widetilde{E}_k, q_0 \vee z_\ra, \ra^{a_2},  \rb^{b_2}, \rc^{c_2} \yields D}
 {\widetilde{E}_{i+1}, \ldots, \widetilde{E}_k, q_0, \ra^{a_2},  \rb^{b_2}, \rc^{c_2} \yields D 
 &
 \widetilde{E}_{i+1}, \ldots, \widetilde{E}_k, z_\ra, \ra^{a_2},  \rb^{b_2}, \rc^{c_2} \yields D}
 \]
Here, again, $a = a_1 + a_2, b = b_1 + b_2, c = c_1 + c_2$. Moreover, derivability of $\widetilde{E}_{i+1}, \ldots, \widetilde{E}_k, z_\ra, \ra^{a_2},  \rb^{b_2}, \rc^{c_2} \yields D$ implies $a_2 = 0$ by statement~\ref{St:z} (induction hypothesis). 

Now by statement~\ref{St:sub} (induction hypothesis) applied to the left premise we conclude that $\Mf$ can move from $\langle q,a,b,c \rangle = \langle q, a_1, b_1 + b_2, c_1 + c_2 \rangle$ to $\langle p', 0, b_2, c_2 \rangle$ in $i-1$ steps. In $\langle p', 0, b_2, c_2 \rangle$ the value of $\ra$ is zero, so applying $\JZDEC(p',\ra,q_0,q')$ changes the configuration to $\langle q_0, 0, b_2, c_2 \rangle$. Finally, applying statement~\ref{St:main} (induction hypothesis) to $\widetilde{E}_{i+1},  \ldots, \widetilde{E}_k, q_0, \rb^{b_2}, \rc^{c_2} \yields D$ shows that $\Mf$ can perform the remaining $k-i$ steps, starting from $\langle q_0, 0, b_2, c_2 \rangle$.
\end{itemize}
This finishes the proof of Lemma~\ref{Lm:omega}.
\end{proof}

\section{$\Sigma_1^0$-Completeness of $\CommACT$}\label{S:ACT}

\subsection{Circular Proofs for Circular Computations}

We start with a reformulation of infinitary commutative action logic, $\CommACTomega$, as a calculus with {\em non-well-founded derivations} (instead of the $\omega$-rule), a commutative variant of the system introduced by Das and Pous~\cite{DasPous2018Action}.
This new calculus, denoted by $\CommACTinfty$, is obtained from $\CommACTomega$ by replacing $*L_\omega$ and $*R_n$ with the following rules:
\[
 \infer[*L]
 {\Gamma, A^* \yields C}
 {\Gamma \yields C & \Gamma, A^*, A \yields C}
 \qquad
 \infer[*R_0]
 {\yields A^*}
 {}
 \qquad
 \infer[*R]
 {\Gamma, \Delta \yields A^*}
 {\Gamma \yields A & \Delta \yields A^*}
\]
Unlike $\CommACTomega$, proofs in $\CommACTinfty$ can be non-well-founded, that is, include infinite branches. These infinite branches should obey the following {\em correctness condition:} on each such branch there should be a {\em trace} of a formula of the form $A^*$, which undergoes $*L$ infinitely many times.

Equivalence between $\CommACTinfty$ and $\CommACTomega$ can be proved in the same way as in the non-commutative case~\cite{DasPous2018Action}. We omit this proof, because $\CommACTinfty$ is not formally used in our complexity arguments; we rather use it to clarify the ideas behind them. 

Reformulation of $\CommACTomega$ as $\CommACTinfty$ makes the construction of the previous section more straightforward: {\em an infinite execution of $\Mf$ is represented by an infinite derivation of $(*)$.}

\begin{example}\label{Ex:INCp}
Let $\Mf$ include the following instruction: $\INC(q_S, \ra, q_S)$. This machine runs infinitely on any input $x$, and this is represented by the following infinite derivation of $(*)$ in $\CommACTinfty$:
\[\small
 \infer[*L]
 {E^*, q_S, \ra^x \yields D}
 {q_S, \ra^x \yields D & 
 \infer=[\wedge L]
 {E^*, E, q_S, \ra^x \yields D}
 {\infer[\imp L]{E^*, q_S \imp (q_S \mul \ra),q_S, \ra^x \yields D}
 {q_S \yields q_S & 
 \infer[\mul L]{E^*, q_S \mul \ra, \ra^x \yields D}
 {\infer{E^*, q_S, \ra^{x+1} \yields D}
 {\infer{\raisebox{4pt}{\vdots}}
 {\infer{E^*, q_S, \ra^{x+2} \yields D}
 {\raisebox{4pt}{\vdots}}}}}}}}
\]
\end{example}

\begin{example}
The zero-check in $\JZDEC$ instantiates an auxiliary infinite branch each time it gets invoked. For example, the infinite run of a machine with $\JZDEC(q_S,\ra,q_S,q_S)$ on input 0 induces the following derivation of $(*)$ with infinitely many infinite branches:
\[
 \small
 \infer[*L]
 {E^*, q_S \yields D}
 {q_S \yields D & 
 \infer=[\wedge L]{E^*, E, q_S \yields D}
 {\infer[\imp L]{E^*, q_S \imp (q_S \vee z_\ra), q_S \yields D}
 {q_S \yields q_S & 
 \infer[\vee L]{E^*, q_S \vee z_\ra \yields D}
 {\infer[*L]{E^*, z_\ra \yields D}
 {z_\ra \yields D & 
 \infer=[\wedge L]{E^*, E, z_\ra \yields D}
 {\infer[\imp L]{E^*, z_\ra \imp z_\ra, z_\ra \yields D}{z_\ra \yields z_\ra & \infer{E^*, z_\ra \yields D}{\raisebox{4pt}{\vdots}}}}}
 & \infer{E^*,q_S \yields D}
 {\infer{\raisebox{4pt}{\vdots}}
 {\infer[\vee L]{E^*, q_S \vee z_\ra \yields D}{\infer{E^*, z_\ra \yields D}{\raisebox{4pt}{\vdots}} & \infer{E^*, q_S \yields D}{\raisebox{4pt}{\vdots}}}}}}
 }}}
\]
\end{example}

Now let us consider a specific class of Minsky computations, namely, {\em circular} ones.

\begin{definition}
 Minsky machine $\Mf$ runs circularly on input $x$, if its execution visits one configuration $\langle p, a,b,c \rangle$ twice (and, due to determinism, infinitely many times, since the sequence of configurations becomes periodic). 
\end{definition}

The key idea is that circular behaviour is represented by {\em circular proofs} of $(*)$ in $\CommACTinfty$:
\[
 \infer[*L]
 {E^*, q_S, \ra^x \yields D}
 {q_S, \ra^x \yields D & 
 \infer{E^*, E, q_S, \ra^x \yields D}
 {\infer{\raisebox{4pt}{\vdots}}
 {\infer[*L]{E^*, p, \ra^a, \rb^b, \rc^c \yields D\tikzmark{B}}
 {p, \ra^a, \rb^b, \rc^c \yields D & 
 \infer{E^*, E, p, \ra^a, \rb^b, \rc^c \yields D}{\infer{\raisebox{4pt}{\vdots}}
 {E^*,p, \ra^a, \rb^b, \rc^c \yields D\tikzmark{A}%
\begin{tikzpicture}[overlay, remember picture, >=latex, distance=-3.0cm]
 \draw[->,in=160,out=160] ($(pic cs:A)+(.1,.1)$) to
 ($(pic cs:B)+(.1,.1)$);
\end{tikzpicture}
 }}}}}}
\]
Here the main infinite branch of our derivation returns to the same sequent, $E^*, p,\ra^a,\rb^b,\rc^c \yields D$, and we may replace further development of the infinite branch by a {\em backlink} to the earlier occurrence of this sequent. Using this backlink, the circular proof can be unravelled into an infinite one.
The correctness condition is guaranteed by the $*L$ rule just above $E^*, p, \ra^a, \rb^b, \rc^c \yields D$: after unravelling, $E^*$ will undergo $*L$ infinitely often.

In fact, circular proofs in $\CommACTinfty$ yield sequents derivable in the narrower logic $\CommACT$ (as defined in Section~\ref{S:upper}).
This can be shown by a commutative modification of the corresponding argument from~\cite{DasPous2018Action}. In this article, however, we perform this translation explicitly for concrete circular derivations used for encoding circular computations. (This is done for simplicity, in order to avoid considering complicated circular proofs with entangled backlinks.)

A similar idea was used to prove undeciability in the non-commutative case, for $\ACT$~\cite{Kuznetsov2019LICS}. In the commutative situation it is even more straightforward, since Minsky computation here is represented directly, without a detour through totality of context-free grammars~\cite{Buszkowski2007,Kuznetsov2019LICS}.

Unfortunately, the translation from circular computations to circular proofs works only in one direction.

\begin{example}
 Let $\Mf$ include the following instruction: $\INC(q_S,\ra,q_S)$. Then the infinite run of $\Mf$, being not a circular one, can be represented by a circular proof of $(*)$ in $\CommACTinfty$. The ``canonical'' proof (see Example~\ref{Ex:INCp}), indeed, is not circular, but there exists an alternative circular one: 
 \[\small
  \infer[Cut]
  {E^*, q_S \yields D}
  {\infer[*L]
  {\tikzmark{BB} E^*, q_S \yields \ra^* \mul q_S}
  {\mbox{\hbox to 10pt{$q_S \yields \ra^* \mul q_S$}} &
  \infer=[\wedge L]
  {E^*, E, q_S \yields \ra^* \mul q_S}
  {\infer[\mul L, \imp L]
  {E^*, q_S \imp (q_S \mul \ra), q_S \yields \ra^* \mul q_S}
  {q_S \yields q_S & \infer[Cut]
  {E^*, q_S, \ra \yields \ra^* \mul q_S}
  {\infer[\mul R]
  {E^*, q_S, \ra \yields \ra \mul (\ra^* \mul q_S)}
  {\tikzmark{AA}
  E^*, q_S \yields \ra^* \mul q_S
  & \ra \yields \ra}
  & \ra \mul (\ra^* \mul q_S) \yields \ra^* \mul q_S}
  }}}
  & 
  \hspace*{-3em}
  \ra^* \mul q_S \yields D}
\begin{tikzpicture}[overlay, remember picture, >=latex, distance=-3.0cm]
 \draw [->] ($(pic cs:AA)+(-.1,.1)$) to[out=0,in=15] ($(pic cs:BB)+(-.05,.1)$);
\end{tikzpicture}
 \]
 where derivations of $q_S \yields \ra^* \mul q_S$, $\ra^* \mul q_S \yields D$, and $\ra \mul (\ra^* \mul q_S) \yields \ra^* \mul q_S$ are obvious.
\end{example}

Thus, we cannot just say ``$\Mf$ runs circularly on $x$ if and only if $(*)$ is derivable in $\CommACT$, and therefore $\CommACT$ is undecidable.'' The ``if'' direction fails. We prove $\Sigma_1^0$-completeness (and, in particular, undecidability) of $\CommACT$ using an indirect technique of {\em effective inseparability,} which we develop in the next subsection. This technique is basically the same as in the non-commutative case~\cite{Kuznetsov2021TOCL}, but the use of Minsky machines instead of Turing ones requires some minor modifications.

\subsection{Effective Inseparability}

The material of this subsection is not new, but rather classical. 
We use the same techniques as in the non-commutative case~\cite{Kuznetsov2021TOCL}; see also~\cite{Rogers,Speranski2016}.
However, we accurately represent these techniques here, because 3-counter machines are quite a restrictive computational model, so we have to ensure that all the constructions work for them as well as for more elaborate computational models, such as Turing machines.

Let 
\begin{align*}
 & \Cc = \{ \langle \Mf, x \rangle \mid \mbox{$\Mf$ runs circularly on $x$} \};\\
 & \Hc = \{ \langle \Mf, x \rangle \mid \mbox{$\Mf$ halts on $x$} \};\\
 & \overline{\Hc} = \{ \langle \Mf, x \rangle \mid \mbox{$\Mf$ does not halt on $x$} \}.
\end{align*}
Obviously, $\Cc \subset \overline{\Hc}$ and $\Hc \cap \overline{\Hc} = \varnothing$. 

In this subsection we are going to show that 
$\Cc$ and $\Hc$ are {\em inseparable:} there is no decidable set $\Kc$ such that $\Cc \subseteq \Kc \subseteq \overline{\Hc}$. Moreover, we shall establish a stronger property of {\em effective inseparability,} from which it will follow, by Myhill's theorem~\cite{Myhill1955}, that if $\Kc$ is r.e.\ and $\Cc \subseteq \Kc \subseteq \overline{\Hc}$, then $\Kc$ is $\Sigma_1^0$-complete. We shall use this result for the set $\Kc(\CommACT) = \{ \langle \Mf, x \rangle \mid \mbox{$(*)$ is derivable in $\CommACT$} \}$ in order to prove that $\CommACT$ is $\Sigma_1^0$-complete.
 
We tacitly suppose that each Minsky machine $\Mf$ is encoded by a natural number $\gn{\Mf}$ in an injective and computable way. We also fix the following encoding of pairs of natural numbers, bijective and computable: 
\[
 [x,y] = \frac{(x+y)(x+y+1)}{2} + x.
\]
Under this convention, 
$\Cc = \{ [\gn{M}, x] \mid \mbox{$\Mf$ runs circularly on $x$} \}$ is a set of natural numbers, and similarly for $\Hc$ and $\overline{\Hc}$.

Let $W_u$, where $u$ is a natural number, be ``the $u$-th r.e.\ set.'' Formally, if $u = \gn{\Mf}$ for some $\Mf$, then $W_u$ is the domain of the function computed by $\Mf$, that is, $W_{\gn{\Mf}} = D_\Mf$; if $u$ does not encode any Minsky machine, let $W_u = \varnothing$.

Now let us define the main notion in our construction, effective inseparability.

\begin{definition}
 Two disjoint sets $A$ and $B$ of natural numbers  are {\em effectively inseparable,} if there exists a computable function $f: \NN \times \NN \to \NN$, such that if $W_u \supseteq A$, $W_v \supseteq B$, and $W_u \cap W_v = \varnothing$, then $f(u,v)$ is defined and $f(u,v) \notin W_u \cup W_v$.
\end{definition}

Effective inseparability yields recursive inseparability, in the following sense:
\begin{prop}
 If $A$ and $B$ are effectively inseparable, then there is no such decidable $K$ that $A \subseteq K$ and $K \cap B = \varnothing$.
\end{prop}

\begin{proof}
 Indeed, if $K$ is decidable, then both $K$ and $\overline{K} = \NN - K$ are r.e. Therefore, $K = W_u$ and $\overline{K} = W_v$ for some $u,v$. We have $W_u \supseteq A$, $W_v \supseteq B$, and $W_u \cap W_v = \varnothing$, but $f(u,v) \notin W_u \cup W_v$ is impossible, since $W_u \cup W_v = \NN$.
\end{proof}

Effective inseparability is closely related to the notion of {\em creative sets,} for which Myhill's theorem holds.

\begin{definition}
 A set $K$ of natural numbers is {\em creative,} if $K$ is r.e.\ and there exists a computable function $h$ such that if $W_u$ is disjoint with $A$, then $h(u)$ is defined and $h(u) \notin W_u \cup K$.
\end{definition}

\begin{theorem}[J.\,Myhill 1955]
 If $K$ is creative, then any r.e.\ set $D$ is m-reducible to $K$. In other words, any creative set is $\Sigma_1^0$-complete.~{\rm\cite{Myhill1955}}
\end{theorem}

Using Myhill's theorem, we prove the result we shall need:

\begin{theorem}\label{Th:effective}
 Let $A$, $B$, and $K$ be three r.e.\ sets of natural numbers such that $A$ and $B$ are effectively inseparable, $A \subseteq K$, and $K \cap B = \varnothing$. Then $K$ is $\Sigma_1^0$-complete.
\end{theorem}

\begin{proof}
We show that $K$ is creative, and then use Myhill's theorem. Since $K$ is r.e., $K = W_v$ for some $v$. For any r.e.\ set $W_u$, which is disjoint with $K$, the set $W_u \cup B$ is also r.e.\ and disjoint with $K$. Moreover, this transformation is effective, i.e., there exists a total computable function $g$ such that $W_u \cup B = W_{g(u)}$. 

Let us define $h(u) = f(v,g(u))$, where $f$ is taken from the definition of effective inseparability of $A$ and $B$. 
Since $W_v = K \supseteq A$, $W_{g(u)} = W_u \cup B \supseteq B$, and $W_v \cap W_{g(u)} = (K \cap W_u) \cup (K \cap B) = \varnothing$, we have $h(u) = f(v,g(u)) \notin W_v \cup W_{g(u)} = K \cup W_u \cup B$. Hence, $h(u) \notin W_u \cup K$. 

This means that $K$ is creative, and therefore it is $\Sigma_1^0$-complete by Myhill's theorem.
\end{proof}

Finally, we show effective inseparability of circular behaviour and halting for Minsky machines:
\begin{theorem}\label{Th:CHinsep}
 $\Cc$ and $\Hc$ are effectively inseparable.
\end{theorem}

In order to prove Theorem~\ref{Th:CHinsep}, we shall need the following technical {\em hardcoding lemma:}

\begin{lemma}\label{Lm:hardcode}
 For any Minsky machine 
 $\Mf$ and any natural number $x$ there exists another Minsky machine $\Mf_x$, such that if $\Mf$ computes $w \mapsto f(w)$ then $\Mf_x$ computes $y \mapsto f([x,y])$. Moreover, the function $x \mapsto \gn{\Mf_x}$ is computable.
\end{lemma}

\begin{proof}
 The new machine $\Mf_x$ is constructed as follows. First we apply the necessary number of $\INC$'s in order to transform $y$ to $x+y$. Now we are in some state $q_1$ with $x+y$ in $\ra$. Second, we include a concrete Minsky machine which computes $z \mapsto z(z+1)/2$. Now we are in another state $q_2$ (the final state of this machine) with $(x+y)(x+y+1)/2$ in $\ra$. Now we again apply $\INC$'s to add $x$, yielding $[x,y] = (x+y)(x+y+1)/2 + x$. Finally, we start $\Mf$.
 
 The dependence on $x$ here is simple: just the number of $\INC$'s in two places in the instruction set. This is clearly computable.
\end{proof}

\begin{proof}[Proof of Theorem~\ref{Th:CHinsep}]
The pairing function is bijective, so we can suppose that any natural number is of the form $[[u,v],w]$. Let us construct a computable function $F$, defined on ``triples'' of the form $[[u,v],w]$. This function will have the following properties, provided  
$W_u \cap W_v = \varnothing$. (If this prerequisite does not hold, the behaviour of $F$ can be arbitrary.)
\begin{enumerate}
 \item If $[w,w] \in W_u$, then $F([[u,v],w])$ is defined and equal to 0.
 \item If $[w,w] \in W_v$, then $F([[u,v],w])$ is defined and equal to 1.
\end{enumerate}

The informal description of the algorithm for $F$ is as follows. It tries (using the universal algorithm) to execute Minsky machines with codes $u$ and $v$ in parallel, on the same  input $[w,w]$. If one of them halts, then $[w,w]$ belongs to $W_u$ or $W_v$ respectively, and we yield the corresponding answer. (Otherwise our algorithm runs forever, and $F$ is undefined. Also, if $u$ or $v$ fails to be a valid code of a Minsky machine, we suppose that its ``execution'' also never stops, thus yielding emptiness of $W_u$ or $W_v$ respectively.)

Since $F$ is computable, it is computed by some Minsky machine (Proposition~\ref{Pr:Turing}), with the final state $q_F$. 
Let us extend this Minsky machine with the following instructions:
\[
 \JZDEC(q_F,\ra,q_{F'},p) 
 \qquad\mbox{and}\qquad
 \JZDEC(p,\ra,p,p),
\]
where states $p$ and $q_{F'}$ are new and $q_{F'}$ is the new final state.

Denote the new machine by $\Nf$. Informally, $\Nf$ does the following: if $[w,w] \in W_u$, then $\Nf$ halts on $w$; if $[w,w] \in W_v$, then $\Nf$ runs circularly on $w$ (it gets stuck in $\langle p, 0, 0,0 \rangle$).

By Lemma~\ref{Lm:hardcode} we can hardcode the first component of the input, $[u,v]$, and obtain a computable function $g\colon (u,v) \mapsto \gn{\Nf_{[u,v]}}$. Finally, let 
$f(u,v) = [g(u,v), g(u,v)] = [\gn{\Nf_{[u,v]}}, \gn{\Nf_{[u,v]}}]$.

This $f$ is clearly computable. Now let us show that if $W_u \supseteq \Cc$, $W_v \supseteq \Hc$, and $W_u \cap W_v = \varnothing$, then $f(u,v) \notin W_u \cup W_v$.
Indeed, if $f(u,v) = [\gn{\Nf_{[u,v]}}, \gn{\Nf_{[u,v]}}] \in W_u$, then $\Nf_{[u,v]}$ halts on input $\gn{\Nf_{[u,v]}}$. This means that $[\gn{\Nf_{[u,v]}}, \gn{\Nf_{[u,v]}}] \in \Hc$, but $\Hc$ is a subset of $W_v$ and therefore is disjoint with $W_u$. Dually, if $f(u,v) \in W_v$, then $\Nf_{[u,v]}$ runs circularly on $\gn{\Nf_{[u,v]}}$, that is, 
$[\gn{\Nf_{[u,v]}}, \gn{\Nf_{[u,v]}}] \in \Cc$. Now $\Cc$ is a subset of $W_u$, and therefore is disjoint with $W_v$. Contradiction.

This finishes the construction of the function needed to establish effective inseparability of $\Cc$ and $\Hc$.
\end{proof}

\subsection{$\Sigma_1^0$-Completeness of $\CommACT$}

Now we formally establish a one-way encoding theorem, from circular Minsky computations to derivability in $\CommACT$.

\begin{theorem}\label{Th:circ}
 If $\Mf$ runs circularly on input $x$, then 
  the sequent
 \[
  E^*, q_S, \ra^x \yields D 
  \eqno{(*)}
 \]
is derivable in $\CommACT$.
\end{theorem}

\begin{proof}
 We start with deriving the ``zero-checking'' sequents $E^*, z_\ra, \rb^b, \rc^c \yields D$, and ditto for $\rb$ and $\rc$. The circular derivation is as follows:
 \[\small
  \infer[*L]
  {E^*, z_\ra, \rb^b, \rc^c \yields D\tikzmark{BX}}
  {\infer=[\vee R]
  {z_\ra, \rb^b, \rc^c \yields D}
  {z_\ra, \rb^b, \rc^c \yields \rb^* \mul \rc^* \mul z_\ra}& 
  \infer=[\wedge L]
  {E^*, E, z_\ra, \rb^b, \rc^c \yields D}
  {\infer[\imp L]{E^*, z_\ra \imp z_\ra, z_\ra, \rb^b, \rc^c \yields D}
  {z_\ra \yields z_\ra & 
  E^*, z_\ra, \rb^b, \rc^c \yields D\tikzmark{AX}}}}
  \begin{tikzpicture}[overlay, remember picture, >=latex, distance=-3.0cm]
 \draw[->,in=160,out=185] ($(pic cs:AX)+(.1,.1)$) to
 %edge[->,bend right=90] 
 ($(pic cs:BX)+(.05,0)$);
\end{tikzpicture}
 \]
 \vskip .3em\noindent
and this is how it gets translated into the calculus for $\CommACT$ presented in Section~\ref{S:upper}:
\[\small
 \infer[\imp R_{\mathrm{inv}}]
 {E^*, z_\ra, \rb^b, \rc^c \yields D}
 {\infer[*L_{\mathrm{ind}}]
 {E^* \yields (z_\ra \mul \rb^b \mul \rc^c) \imp D}
 {\infer=[\mul L, \imp R]
 {\yields (z_\ra \mul \rb^b \mul \rc^c) \imp D}
 {\infer=[\vee R]{z_\ra, \rb^b, \rc^c \imp D}
 {z_\ra, \rb^b, \rc^c \yields \rb^* \mul \ra^* \mul z_\ra}}
 & 
 \infer=[\mul L, \imp R]
 {E, (z_\ra \mul \rb^b \mul \rc^c) \imp D \yields (z_\ra \mul \rb^b \mul \rc^c) \imp D}
 {\infer=[\wedge L]{E, z_\ra, \rb^b, \rc^c, (z_\ra \mul \rb^b \mul \rc^c) \imp D \yields D}{\infer[\imp L]{z_\ra \imp z_\ra, z_\ra, \rb^b, \rc^c, (z_\ra \mul \rb^b \mul \rc^c) \imp D \yields D}{\infer[\imp L]{z_\ra \imp z_\ra, z_\ra, \rb^b, \rc^c \yields z_\ra \mul \rb^b \mul \rc^c}{z_\ra \yields z_\ra & z_\ra, \rb^b, \rc^c \yields z_\ra \mul \rb^b \mul \rc^c} & D \yields D}}}}}
\]
(Here and further $\imp R_{\mathrm{inv}}$ is the inversion of a series of $\imp R$ applications, which is established by cut with $z_\ra, \rb^b, \rc^c,
(z_\ra \mul \rb^b \mul \rc^c) \imp D \yields D$.)

Next, we produce a circular derivation of $(*)$ in the following way. Each step of $\Mf$'s execution, $\langle p,a,b,c \rangle \to \langle q,a',b',c' \rangle$, can be represented as a subderivation with $E^*, p, \ra^a, \rb^b, \rc^c \yields D$ as the goal and $E^*, q, \ra^{a'}, \rb^{b'}, \rc^{c'} \yields D$ as a hypothesis. Moreover, the lowermost rule in this derivation is $*L$, which makes the correctness condition valid.

The construction is basically the same as in the proof of Lemma~\ref{Lm:omega}. For $\INC(p,\ra,q)$ we have
\[\small
 \infer[*L]
 {E^*, p, \ra^a, \rb^c, \rc^c \yields D}
 {p, \ra^a, \rb^b, \rc^c \yields D & 
 \infer=[\wedge L]
 {E^*, E, p, \ra^a, \rb^b, \rc^c \yields D}
 {\infer[\imp L]{E^*, p \imp (q \mul \ra), p, \ra^a, \rb^b, \rc^c \yields D}
 {p \yields p & \infer[\mul L]{E^*, q \mul \ra, \ra^a, \rb^b, \rc^c \yields D}
 {E^*, q, \ra^{a+1}, \rb^b, \rc^c \yields D}}}}
\]
For $\JZDEC(p,\ra,q_0,q_1)$ and $a \ne 0$,
\[\small
 \infer[*L]
 {E^*, p, \ra^a, \rb^c, \rc^c \yields D}
 {p, \ra^a, \rb^b, \rc^c \yields D & 
 \infer=[\wedge L]
 {E^*, E, p, \ra^a, \rb^b, \rc^c \yields D}
 {\infer[\imp L]{E^*, (p \mul \ra) \imp q_1, p, \ra^a, \rb^b, \rc^c \yields D}
 {\infer[\mul R]{p, \ra \yields p \mul \ra}{p \yields p & \ra \yields \ra} & 
 E^*, q_1, \ra^{a-1}, \rb^b, \rc^c \yields D
 }}}
\]
Finally, for $\JZDEC(p,\ra,q_0,q_1)$ and $a = 0$ we have
\[\small
 \infer[*L]
 {E^*, p, \rb^c, \rc^c \yields D}
 {p, \rb^b, \rc^c \yields D & 
 \infer=[\wedge L]
 {E^*, E, p, \rb^c, \rc^c \yields D}
 {\infer[\imp L]{E^*, p \imp (q_0 \vee z_\ra), p, \rb^b, \rc^c \yields D}
 {p \yields p & 
 \infer[\vee L]{E^*, q_0 \vee z_\ra, \rb^b, \rc^c \yields D}{E^*, q_0, \rb^b, \rc^c \yields D & E^*, z_\ra, \rb^b, \rc^c \yields D}}}
 }
\]
Here $\ra^a, \rb^b, \rc^c, p \yields D$ (in particular, $\rb^b, \rc^c, p \yields D$) is easily derivable using $*R$, $\mul R$, and $\vee R$, and derivability of $E^*, z_\ra, \rb^b, \rc^c \yields D$ in $\CommACT$ was established earlier.

Next, we connect these subderivations in order to represent the infinite run of $\Mf$. Since this run is circular, a sequent of the form $E^*, p, \ra^a, \rb^b, \rc^c \yields D$ gets repeated, and we arrive at a circular proof:
\[\small
 \infer[*L]
 {E^*, q_S \yields D}
 {q_S \yields D & 
 \infer{E^*, E, q_S \yields D}
 {\infer{\raisebox{4pt}{\vdots}}
 {\infer[*L]{E^*, p, \ra^a, \rb^b, \rc^c \yields D\tikzmark{BZ}}
 {p, \ra^a, \rb^b, \rc^c \yields D & 
 \infer{E^*, E, p, \ra^a, \rb^b, \rc^c \yields D}{\infer{\raisebox{4pt}{\vdots}}
 {E^*,p, \ra^a, \rb^b, \rc^c \yields D\tikzmark{AZ}%
\begin{tikzpicture}[overlay, remember picture, >=latex, distance=-3.0cm]
 \draw[->,in=160,out=160] ($(pic cs:AZ)+(.1,.1)$) to
 %edge[->,bend right=90] 
 ($(pic cs:BZ)+(.1,.1)$);
\end{tikzpicture}
 }}}}}}
\]

Now we translate this circular proof into $\CommACT$.
Let us first consider the upper part of this proof, 
\[\small\infer[*L]{E^*, p, \ra^a, \rb^b, \rc^c \yields D\tikzmark{BZZ}}
 {p, \ra^a, \rb^b, \rc^c \yields D & 
 \infer{E^*, E, p, \ra^a, \rb^b, \rc^c \yields D}{\infer{\raisebox{4pt}{\vdots}}
 {E^*,p, \ra^a, \rb^b, \rc^c \yields D\tikzmark{AZZ}%
\begin{tikzpicture}[overlay, remember picture, >=latex, distance=-3.0cm]
 \draw[->,in=160,out=160] ($(pic cs:AZZ)+(.1,.1)$) to
 %edge[->,bend right=90] 
 ($(pic cs:BZZ)+(.1,.1)$);
\end{tikzpicture}
 }}} 
\]
and translate it into a proof in $\CommACT$ using $*L_{\mathrm{ind}}$. 

Let $k$ be the number of $*L$ applications on the main branch; $k \ge 1$. First, from this circular derivation we can extract derivations of $E^i, p, \ra^a, \rb^b, \rc^c \yields D$ for $0 \le i < k$. Indeed, we replace $E^*$ with $E^i$ in the goal and then proceed upwards, choosing the right branch of the first $i$ applications of $*L$ (at each step $i$ gets decreased by 1) and the left branch at the $(i+1)$-st one.

Second, we derive $E^k, (p \mul \ra^a \mul \rb^b \mul \rc^c) \imp D, p, \ra^a, \rb^b, \rc^c \yields D$. Here we replace $E^*$ with $E^k, (p \mul \ra^a \mul \rb^b \mul \rc^c) \imp D$ in the goal sequent, and then always choose the right branch, and in the end $k$ reduces to 0, and we enjoy a derivable sequent
$(p \mul \ra^a \mul \rb^b \mul \rc^c) \imp D, p, \ra^a, \rb^b, \rc^c \yields D$ instead of the backlinked $E^*, p, \ra^a, \rb^b, \rc^c \yields D$.

Now we glue everything up. The desired sequent $E^*, p, \ra^a, \rb^b, \rc^c \yields D$ is obtained by $Cut$ from $E^* \yields (\bigvee_{i=0}^{k-1} E^i) \mul (E^k)^*$ and
$(\bigvee_{i=0}^{k-1} E^i) \mul (E^k)^*, p, \ra^a, \rb^b, \rc^c \yields D$. The former is a well-known principle of Kleene algebra, which is derivable in $\ACT$ and therefore in $\CommACT$. The derivation of the latter is presented below, where, for brevity, $\Gamma = p, \ra^a, \rb^b, \rc^c$, $\bullet\Gamma = p \mul \ra^a \mul \rb^b \mul \rc^c$:
\[\small
\infer[\mul L]{
(\bigvee_{i=0}^{k-1} E^i) \mul (E^k)^*, \Gamma \yields D}
{\infer[\imp R_{\mathrm{inv}}]
{\bigvee_{i=0}^{k-1} E^i,  (E^k)^*, \Gamma \yields D}
{\infer[*L_\mathrm{ind}]{(E^k)^* \yields (\bigvee_{i=0}^{k-1} E^i) \imp ({\bullet\Gamma} \imp D)}
{
\infer=[\mul L,\imp R]
{\yields (\bigvee_{i=0}^{k-1} E^i) \imp ({\bullet\Gamma} \imp D)}
{\infer=[\vee L]{\bigvee_{i=0}^{k-1} E^i, \Gamma \yields D}
{\ldots & E^i, \Gamma \yields D & \ldots}}
&
\infer=[\mul L, \imp R]
{E^k, (\bigvee_{i=0}^{k-1} E^i) \imp ({\bullet\Gamma} \imp D) \yields 
(\bigvee_{i=0}^{k-1} E^i) \imp ({\bullet \Gamma} \imp D)}
{\infer[\imp L]{E^k, \Gamma, \bigvee_{i=0}^{k-1} E^i, (\bigvee_{i=0}^{k-1} E^i) \imp ({\bullet\Gamma} \imp D) \yields D}
{\bigvee_{i=0}^{k-1} E^i \yields \bigvee_{i=0}^{k-1} E^i & 
E^k, \Gamma, {\bullet\Gamma} \imp D \yields D}}}}}
%}
\]

Now we have a finite, non-circular derivation of our main goal $E^*, q_S \yields D$ in a combined system, which includes both axioms and rules of $\CommACT$ (as defined in Section~\ref{S:upper}) and the $*L$ rule of $\CommACTinfty$. We finish our argument by showing that $*L$ is derivable in $\CommACT$:
\[\small
 \infer[Cut]
 {\Gamma, A^* \yields C}
 {A^* \yields \U \vee (A \mul A^*) &
 \infer[\vee L]{\Gamma, \U \vee (A \mul A^*) \yields C}
 {\infer[\U L]{\Gamma, \U \yields C}{\Gamma \yields C} & 
 \infer[\mul L]{\Gamma, A \mul A^* \yields C}{\Gamma, A^*, A \yields C}}}
\]
Here $A^*  \yields \U \vee (A \mul A^*)$ is again a principle of Kleene algebra, which is derivable in $\ACT$ and therefore in $\CommACT$.
\end{proof}

Now we proceed exactly as in the non-commutative case~\cite{Kuznetsov2021TOCL}. Let 
\begin{align*}
 & \Kc(\CommACTomega) = \{ \Mf \mid \mbox{$(*)$ is derivable in $\CommACTomega$} \};\\
 & \Kc(\CommACT) = \{ \Mf \mid \mbox{$(*)$ is derivable in $\CommACT$} \}
\end{align*}
and recall that
\begin{align*}
 & \Cc = \{ \Mf \mid \mbox{$\Mf$ runs circularly} \}; \\
 & \overline{\Hc} = \{ \Mf \mid \mbox{$\Mf$ does not halt} \}. \\
\end{align*}
By Theorem~\ref{Th:omega} and Theorem~\ref{Th:circ} we have
\[
 \Cc \subset \Kc(\CommACT) \subset \Kc(\CommACTomega) = \overline{\Hc}.
\]
Now Theorem~\ref{Th:CHinsep} and  Theorem~\ref{Th:effective} immediately yield $\Sigma_1^0$-completeness of $\Kc(\CommACT)$ and, therefore, of $\CommACT$ itself:

\begin{theorem}
 The derivability problem for $\CommACT$ is $\Sigma_1^0$-complete.
\end{theorem}

\section{Conclusion}\label{S:conclusion}
In this article, we have established $\Pi_1^0$-completeness of $\CommACTinfty$ and $\Sigma_1^0$-completeness of $\CommACT$. The former is a commutative counterpart of results by Buszkowski and Palka~\cite{Buszkowski2007,Palka2007}. The latter is a commutative counterpart of an earlier result by the author~\cite{Kuznetsov2019LICS,Kuznetsov2021TOCL}. 

In fact, as in the non-commutative case, we have established $\Sigma_1^0$-completeness not only for $\CommACT$, but for a range of logics: namely, any r.e.\ logic between $\CommACT$ and $\CommACTomega$ is $\Sigma_1^0$-complete.

There are several questions are left for further research:
\begin{enumerate}
 \item The complexity question for $\CommACT$ without additive connectives ($\wedge$ and $\vee$) is open. Notice that additives are crucial for encoding the $\JZDEC$ instruction; in~\cite{LMSS}, there is no $\JZDEC$, but there are parallel computations, also simulated using $\vee$.
\item It is an open question whether the same complexity results hold for the variants of $\CommACTomega$ and $\CommACT$ with distributivity of $\vee$ over $\wedge$ added.
\item The complexity of the Horn theory for commutative action lattices or even commutative Kleene algebras is, to the best of the author's knowledge, unknown. Comparing with Kozen's result for non-commutative Kleene algebras~\cite{Kozen2002}, we conjecture $\Pi_1^1$-completeness for the *-continuous case, and the proof should again use Minsky machines instead of Turing ones.
\item It is also interesting to look at the non-associative, but commutative, version of infinitary action logic. In the non-associative case, it is problematic to define iteration, and it gets replaced with so-called {\em iterative division,} that is, compound connectives of the form $A^* \imp B$ and $B \pmi A^*$. The interesting phenomenon here is that the corresponding non-commutative system happens to be algorithmically decidable, at least with the distributivity axiom added~\cite{Sedlar2020}. On the other hand, as shown in~\cite{KuznetsovRyzhkova2020}, in the associative and non-commutative case iterative division is sufficient for $\Pi_1^0$-hardness.
\end{enumerate}

\subsection*{Acknowledgments}

The author is grateful to the participants of the DaL\'{\i} 2020 online meeting for fruitful discussions, especially on directions of further research. 

\paragraph*{Financial Support.} The work was supported by the Russian Science Foundation, in cooperation with the Austrian Science Fund, under grant RSF–FWF 20-41-05002.

\bibliographystyle{abbrv}
\bibliography{ACT}

\end{document}